\documentclass[reqno,12pt]{amsart}
\usepackage[pagebackref]{hyperref}
\usepackage[a4paper,margin=2.5cm]{geometry}

\newcommand{\weak}{\rightsquigarrow}

\newcommand{\Var}{\mathrm{Var}\,}
\newcommand{\Med}{\mathrm{Med}\,}
\newcommand{\E}{\mathbb{E}}
\newcommand{\p}{\mathbb{P}}

\newcommand{\C}{\mathbb{C}}

\newcommand*{\ind}[1]{\mathbf{1}_{\{#1\}}}
\newcommand*{\Ind}[1]{\mathbf{1}_{#1}}

\newcommand{\bfx}{\mathbf{x}}

\newcommand{\R}{\mathbb{R}}

\newcommand{\id}{\mathrm{Id}}
\newcommand{\tr}{\mathrm{tr}}

\newcommand{\INCOMP}{\mathrm{Incomp}}
\newcommand{\COMP}{\mathrm{Comp}}
\newcommand{\SPARSE}{\mathrm{Sparse}}

\newcommand{\DIST}{\mathrm{dist}}
\newcommand{\rank}{\mathrm{rank}\;}
\newtheorem{lemma}{Lemma}[section]
\newtheorem{theorem}[lemma]{Theorem}

\newtheorem{remark}[lemma]{Remark}
\newtheorem{corollary}[lemma]{Corollary}

\title{Circular law for random matrices with~exchangeable~entries}

\thanks{Support: Polish Ministry of Science and Higher Education Iuventus Plus
  Grant no. IP 2011 000171, and French ANR-2011-BS01-007-01 GeMeCoD and
  ANR-08-BLAN-0311-01 Granma.}
  
\author{Rados{\l}aw Adamczak}%
\address[RA]{Institute of Mathematics, University of Warsaw, ul. Banacha 2,
  02-097 Warszawa, POLAND.}

\author{Djali{\l} Chafa\"{\i}} %
\address[DC]{Ceremade, Universit\'e Paris-Dauphine, Place du Mar\'echal de
  Lattre de Tassigny, 75775 Paris Cedex 16, FRANCE; and Institut Universitaire
  de France, FRANCE.}

\author{Pawe{\l} Wolff} %
\address[PW]{Institute of Mathematics, University
  of Warsaw, ul. Banacha 2, 02-097 Warszawa, POLAND; and Institute of Mathematics, Polish Academy of Sciences, ul. \'Sniadeckich 8, 00-956 Warszawa, POLAND.}

\date{Winter 2014, compiled on \today.}

\begin{document}

\begin{abstract}
  An exchangeable random matrix is a random matrix with distribution invariant
  under any permutation of the entries. For such random matrices, we show, as
  the dimension tends to infinity, that the empirical spectral distribution
  tends to the uniform law on the unit disc. This is an instance of the
  universality phenomenon known as the circular law, for a model of random
  matrices with dependent entries, rows, and columns. It is also a
  non-Hermitian counterpart of a result of Chatterjee on the semi-circular law
  for random Hermitian matrices with exchangeable entries. The proof relies in
  particular on a reduction to a simpler model given by a random shuffle of a
  rigid deterministic matrix, on Hermitization, and also on combinatorial
  concentration of measure and combinatorial Central Limit Theorem. A crucial
  step is a polynomial bound on the smallest singular value of exchangeable
  random matrices, which may be of independent interest.
\end{abstract}


\maketitle

{\small\tableofcontents}

\section{Introduction}

The spectral measure of an $n\times n$ matrix $A$ with real or complex entries
is defined by
\begin{displaymath}
\nu_A = \frac{1}{n}\sum_{k=1}^n \delta_{\lambda_k},
\end{displaymath}
where $\lambda_1,\ldots,\lambda_n$ are the eigenvalues of $A$, in other words
the roots in $\C$ of the characteristic polynomial, counted with their
algebraic multiplicities, and where $\delta_x$ is the Dirac mass at the point
$x$. Thus $\nu_A$ is a Borel probability measure on $\C$, supported on $\R$ if
$A$ is Hermitian. If $A$ is a random matrix, then $\nu_A$ is a random discrete
probability measure.

An extensive body of work has been devoted to the study of the behavior of the
spectral measure of large dimensional random matrices. While Wigner
\cite{Wigner} considered Hermitian matrices, the non-Hermitian case has also
attracted attention, starting with the early work by Mehta \cite{Mehta1}, who
proved, by using explicit formulas due to Ginibre \cite{Ginibre}, that the
average spectral measure of $n\times n$ matrices with i.i.d.\ standard complex
Gaussian entries scaled by $\sqrt{n}$ converges to the uniform measure on the
unit disc. We will call this measure the \emph{circular law}. In a series of
subsequent developments \cite{G1,G2,Bai,PanZhou,GTCirc}, culminating with the
work of Tao and Vu \cite{TV}, this result has been strengthened to almost sure
convergence and to sequences of random matrices with i.i.d.\ entries of unit
variance. The Tao-Vu theorem is an instance of the \emph{universality
  phenomenon}, which has become one of the main themes of Random Matrix
Theory. Namely the limiting spectral object depends on the actual distribution
of the entries of the random matrix only via a global scale parameter such as
the variance. It is natural to ask whether this universal behavior extends to
models of random matrices with heavy-tailed entries or with dependent entries.
It has been recently shown \cite{BCC2,BCAround} that the spectral measure of
matrices with independent entries in the domain of attraction of an
$\alpha$-stable distribution with $0<\alpha<2$ converges to a measure which
depends only on $\alpha$ and is supported on the whole complex plane. In
particular the limiting distribution is not the circular law. Regarding the
independence assumption, several classes of random matrices have been
discussed in the literature, starting from models with i.i.d.\ rows such as
random Markov matrices \cite{BCC1}, random matrices with i.i.d.\ log-concave
rows \cite{CircLawUnc,InfoNoise}, and random $\pm1$ matrices with i.i.d.\ rows
of given sum \cite{NguyenVuCirc}. The circular law still holds for models with
dependent rows, for instance for blocs of Haar unitary matrices
\cite{MR2480790,MR2919538}, for random doubly stochastic matrices following
the uniform distribution on the Birkhoff polytope \cite{NguyenDoublyStoch},
for random matrices with log-concave unconditional distribution
\cite{RADC_CircUnc}, and for random matrices with i.i.d.\ quaternionic entries
\cite{NguyenORourkeQuaternion} (seen as $2\times 2$ i.i.d.\ real blocks with
possible dependencies inside the blocks).

The present work introduces first a model of random matrices with exchangeable
entries, obtained by shuffling a globally constrained deterministic matrix
using a random uniform permutation. Such random matrices have dependent
entries, rows, and columns. We show that the asymptotic behavior of the
spectral measure is governed by the circular law. The precise formulation is
given in Theorem \ref{th:circular_law} below. This first result is then used
to deduce the circular law for a larger class of random matrices with
exchangeable entries, as stated in Theorem \ref{th:circular_law_2} below. This
second result can be seen as a counterpart of the one of Chatterjee
\cite{ChatterjeeLindeberg}, who proved that the spectral measure of symmetric
random matrices with exchangeable entries in the upper triangle converges to
Wigner's semi-circular law, which is the universal limit for random matrices
with i.i.d.\ entries of finite variance. Our model of random matrices with
exchangeable entries, like all the available models belonging to the
universality class of the circular law, has symmetries related to the
canonical basis. On the other hand, we stress that the sole invariance under
permutation of rows and columns, which is weaker than full exchangeability, is
not sufficient, as shown by the simple example of Haar unitary random matrices
for which the limiting spectral distribution is the uniform distribution on
the unit circle (not the unit disc!) often referred to as the arc law
\cite{MR2919538}. On the other hand, note that the uniform law on the Birkhoff
polytope of doubly stochastic matrices is invariant by permutations of rows or
columns, but is not exchangeable globally.

A main tool in our proof is a lower bound on the smallest singular value of
random matrices with exchangeable entries, formulated in Theorem
\ref{th:ssv}. In recent years, an extensive amount of work has been devoted
to the analysis of the smallest singular value of random matrices with
(partially) independent entries or rows, in connection with the circular law,
but also with numerical analysis or geometry, which makes our result of
independent interest. Our approach follows the path introduced by Rudelson and
Vershynin in \cite{RV} for random matrices with i.i.d.\ entries, more
precisely what they call the ``soft approach'' (nothing more is needed for the
purpose of the circular law). However several technical problems related to the lack of independence among the entries have to be
overcome.

\subsection*{Main results}

For any $n\geq1$, let $\bfx^{(n)} = (\bfx_{ij}^{(n)})_{1\le i,j\le n}$ be a
deterministic real matrix, with
\begin{itemize}
\item[(A1)] $\sum_{i,j=1}^n \bfx^{(n)}_{ij} = 0$;
\item[(A2)] $\sum_{i,j=1}^n |\bfx^{(n)}_{ij}|^2 = n^2$.
\end{itemize}
Thanks to assumption (A2) we have
\begin{align}\label{def:K}
  K_n:=\max_{1\leq i,j\leq n}|\bfx^{(n)}_{ij}|\geq1.
\end{align}

Of course, the entries of $\bfx^{(n)}$ depend typically on $n$, hence the
notation. The assumption (A2) puts also a constraint on the sparsity of the
vector $\bfx^{(n)}$, depending on $K_n$.

Let now $\pi_n$ be a uniform
random permutation of the set $[n]\times [n]=\{(i,j):1\leq i,j\leq n\}$, where
$[n] = \{1,\ldots,n\}$. We consider the random matrix
\begin{equation}\label{eq:model}
  X^{(n)} = (X_{ij}^{(n)})_{1\le i,j\le n}
  \quad\text{where}\quad
  X^{(n)}_{ij} := \bfx^{(n)}_{\pi_n(i,j)}.
\end{equation}
The matrix $X^{(n)}$ inherits the structure of $\bfx^{(n)}$. Namely, from (A1-A2) we get
\[
\sum_{i,j=1}^nX^{(n)}_{ij}=0
\quad\text{and}\quad
\sum_{i,j=1}^n|X^{(n)}_{ij}|^2=n^2
\quad\text{and}\quad
\max_{1\leq i,j\leq n}|X^{(n)}_{ij}|=K_n.
\]
Since the uniform law on the symmetric group is invariant under translation (it
is a normalized Haar measure), it follows that the law of $X^{(n)}$ is
invariant under any deterministic permutation of its $n^2$ coordinates, i.e. the random vector $(X^{(n)}_{ij}:(i,j)\in [n]\times[n])$ of
$\R^{n^2}$ has exchangeable coordinates. In other words, for every permutation
$\sigma$ of $[n]\times[n]$,
\[
{(X^{(n)}_{\sigma(i,j)})}_{1\leq i,j\leq n} %
=
{(\bfx^{(n)}_{(\pi_n \circ\sigma)(i,j)})}_{1\leq i,j\leq n}
\overset{d}{=} %
{(\bfx^{(n)}_{\pi_n(i,j)})}_{1\leq i,j\leq n}
= %
{(X^{(n)}_{i,j})}_{1\leq i,j\leq n}.
\]
It follows that for every
$(i,j)\in[n]\times[n]$, and for every $(k,l)\in[n]\times[n]$ with
$(i,j)\neq(k,l)$,
\begin{equation}\label{eq:xmeancov}
  \E(X^{(n)}_{ij})=0
  \quad\text{and}\quad
  \E(|X^{(n)}_{ij}|^2)=1
  \quad\text{and}\quad
  \E(X^{(n)}_{ij}X^{(n)}_{kl})=-\frac{1}{n^2-1}.
\end{equation}
Except for special choices of $\bfx^{(n)}$, the random matrix $X^{(n)}$ is
non-Hermitian, and even non-normal with high probability. Normality means here
commutation with the transpose-conjugate. The random matrix $X^{(n)}$ has
exchangeable entries obtained by shuffling the globally constrained
deterministic matrix $\bfx^{(n)}$ using the random uniform permutation
$\pi_n$. Such random matrices have dependent entries, dependent rows, and
dependent columns. The simulation of a random uniform permutation such as
$\pi_n$ amouts to approximately $2n^2\log_2(n)$ bits. The randomness in
$X^{(n)}$ is the one of $\pi_n$, which should be compared with the $n^2$ bits
needed for a standard random $n\times n$ symmetric Bernoulli $\pm1$ matrix.

Our first main result is a bound on the smallest
singular value of the matrix $X^{(n)}$. For an $n\times n$ real or complex
matrix $A$ we denote by
\[
s_1(A) \ge \cdots \ge s_n(A)
\]
the singular values of $A$, i.e. eigenvalues of $\sqrt{AA^\ast}$ where
$A^*=\overline{A}^\top$ (transpose-conjugate). The operator norm of $A$ is
$\|A\|=\max_{|x|=1}|Ax|=s_1(A)$, where $|x|$ denotes the Euclidean norm in
$\C^n$ or in $\R^n$. We have $s_n(A)=\min_{|x|=1}|Ax|$. The matrix $A$ is
invertible iff $s_n(A)>0$ and in this case $s_n(A)=\|A^{-1}\|^{-1}$. The
Hilbert-Schmidt norm or the Frobenius norm of $A$ is
$\|A\|_{\mathrm{HS}}=\sqrt{\tr(AA^*)}=\sqrt{s_1(A)^2+\cdots+s_n(A)^2}$.

\begin{theorem}[Smallest singular value]\label{th:ssv}
  If $X^{(n)}$ is as in \eqref{eq:model} then for every $z\in\C$, there exists
  a constant $C_{K_n,z}>0$ depending only on $K_n$ (defined in \eqref{def:K}) and $z$ such that for every
  $\varepsilon>0$,
  \begin{displaymath}
    \p\Big(s_n(X^{(n)} - z\sqrt{n}\id) \le \varepsilon n^{-1/2}\Big) %
    \le C_{K_n,z}\Big(\varepsilon + \frac{1}{n^{1/2}}\Big).
  \end{displaymath}
\end{theorem}

An explicit dependence of $C_{K,z}$ on $K$ and $z$ is given in Theorem
\ref{th:ssv_detailed}.

Let ${(\nu_n)}_{n\geq1}$ be a sequence of random probability measures on $E$
with $E=\C$ or $E=\R$, and let $\nu$ be a probability measure on $E$. We say
that the sequence ${(\nu_n)}_{n\geq1}$ tends as $n\to\infty$ weakly in
probability to $\nu$, and we denote
\[
\nu_n\underset{n\to\infty}{\weak}\nu,
\]
if for any $f\in\mathcal{C}(E,\R)$ and $\varepsilon>0$,
\[
\lim_{n\to\infty}\mathbb{P}
\left(\left|\int\!f\,d\nu_n-\int\!f\,d\nu\right|\geq\varepsilon\right)
=0,
\]
where $\mathcal{C}(E,\R)$ is the class of all bounded continuous real valued
functions on $E$. This is equivalent to saying that ${(\nu_n)}_{n\geq1}$, seen
as a sequence of random variables in the space of Borel probability measures
on $E$, converges in probability to $\nu$ (see \cite[Definition 1.2]{TV}).

We denote by $\nu^{\mathrm{circ}}$ the uniform law on the unit disk of $\C$,
with density $z\mapsto\pi^{-1}\mathbf{1}_{\{z\in\C:|z|\leq1\}}$. Our second
main result concerns the limiting behavior of the spectral measure of
$X^{(n)}$.

\begin{theorem}[Circular law for shuffled matrices]\label{th:circular_law}
  Let $A^{(n)}=\frac{1}{\sqrt{n}}X^{(n)}$ where $X^{(n)}$ is as in
  \eqref{eq:model}. If $K_n = \mathcal{O}(n^{1/(10+\delta)})$ for some $\delta
  > 0$, then $\nu_{A^{(n)}}\underset{n\to\infty}{\weak}\nu^{\mathrm{circ}}$.
\end{theorem}

One can check easily that the normalization by $n^{-1/2}$ ensures tightness,
namely

\begin{multline*}
  \int\!|\lambda|^2\,d\nu_{A^{(n)}}(\lambda)
  =\frac{1}{n}\sum_{k=1}^n|\lambda_k(A^{(n)})|^2 \\
  \leq
  \frac{1}{n^2}\sum_{k=1}^ns_k(X^{(n)})^2
  % \int\!s^2\,d\nu_{A^{(n)A^{(n)*}}}(s)
  =\frac{1}{n^2}\|X^{(n)}\|_{\mathrm{HS}}^2
  =\frac{1}{n^2}\sum_{i,j=1}^n|X^{(n)}_{ij}|^2
  =1.
\end{multline*}

Theorem \ref{th:circular_law} leads to the following result concerning more
general random matrices with exchangeable entries, beyond model
\eqref{eq:model}. Namely, for every $n\gg1$, let
\begin{equation}\label{eq:model2}
  Y^{(n)} = (Y^{(n)}_{ij})_{1\le i,j \le n}
\end{equation}
be an $n\times n$ matrix with exchangeable
real entries, and define
\begin{displaymath}
  \mu_n := \frac{1}{n^2}\sum_{i,j=1}^n Y^{(n)}_{ij}%
  \quad\text{and}\quad%
  \sigma_n^2 := \frac{1}{n^2}\sum_{i,j=1}^n (Y^{(n)}_{ij} - \mu_n)^2.
\end{displaymath}

\begin{theorem}[Circular law for matrices with exchangeable
  entries]\label{th:circular_law_2}
  Let $Y^{(n)}$ be as in \eqref{eq:model2}. If for $n\gg1$,
  $\sigma_n^2 > 0$ a.s. and for some $\delta > 0$,
  \[
  \limsup_{n\to \infty}
  \E\Big[\left(\sigma_n^{-1}|Y^{(n)}_{11}-\mu_n|\right)^{20+\delta}\Big]<\infty,
  \]
  then, denoting
  \[
  B^{(n)} = \frac{1}{\sqrt{n}\sigma_n} (Y^{(n)}_{ij} - \mu_n)_{1\le i,j\le n},
  \]
  we have
  \[
  \nu_{B^{(n)}}\underset{n\to\infty}{\weak}\nu^{\mathrm{circ}}.
  \]
\end{theorem}

\subsection*{Discussion and open problems}

\subsubsection*{Assumptions.} It is not enough to assume in Theorem
\ref{th:circular_law} that $K_n = \mathcal{O}(\sqrt{n})$, since one could
construct a matrix with at most $n/2$ nonzero entries, which would create a
substantial atom at zero in the spectral measure. It is tempting to conjecture
that $K_n = \mathcal{O}(n^{1/(2+\delta)})$ is sufficient for the circular law
to hold. This would allow to prove Theorem \ref{th:circular_law_2} under the
assumption on moments of order $4+\delta$. On the other hand, Chatterjee's
counterpart \cite{ChatterjeeLindeberg} to Theorem \ref{th:circular_law_2} in the Hermitian case works
under the assumption which when translated to our setting would be $\E
(|Y^{(n)}_{11} - \mu_n|/\sigma_n)^{4} = o(n^{2/3})$.

\subsubsection*{Operator norm.}

One way to improve our growth/integrability assumptions would be to obtain a
better bound on the operator norm of $X^{(n)}$ defined in \eqref{eq:model}.
The simple bound we use (Lemma \ref{le:operator_norm}) is $\E\|X^{(n)}\|\le
CK_n\sqrt{n}$. By analogy with the available results for matrices with i.i.d.\
entries, one may ask if for some $\alpha < 1/2$ and a universal constant $C$,
\begin{displaymath}
  \E\|X^{(n)}\| \overset{?}{\le} \sqrt{2n} +CKn^{\alpha}.
\end{displaymath}
Another way of weakening our assumptions would be to improve the factor
$1/\sqrt{n}$ in the probability bound on the smallest singular value. This
factor comes from the Combinatorial Central Limit Theorem and cannot be
further improved by our methods. In the i.i.d. case it has been strengthened
to $e^{-cn}$ by relating the bound on the smallest singular value to the
Littlewood-Offord problem \cite{RV,TV1,MR2507275}. While certain ingredients
of this approach can be transferred to the exchangeable setting, an important
step seems to rely on a Fourier analytic argument, related to Esseen's
inequality, involving independence of the entries of the matrix. It would be
interesting to develop tools which would allow to use similar ideas in the
exchangeable case. Note that since the right hand side in Theorem
\ref{th:ssv} is not summable in $n$, one cannot use the first Borel-Cantelli
lemma in order to get an almost sure lower bound on
$s_n(X^{(n)}-z\sqrt{n}\id)$.

\subsubsection*{Other models.}
One can consider a model of the form $M_n = n^{-1/2}(X^{(n)} + F_n)$, where
$F_n$ is deterministic with $\rank F_n = o(n)$, satisfying a bound on the
operator norm. An inspection of the proof of Theorem \ref{th:circular_law}
reveals that if $\|F_n\| =\mathcal{O}(\sqrt{n})$, then one can obtain the
convergence to the circular law for the spectral measure of $M_n$ without
changing the assumptions on $K_n$. On the other hand if $\|F_n\| =
o(n/\log^{3/2}n)$, then an adaptation of our proof gives the result for
$K_n=\mathcal{O}(1)$. We remark that for matrices with i.i.d. entries with
mean zero and variance one it is enough to assume that $\|F_n\|_{\mathrm{HS}}
=\mathcal{O}(n)$ \cite{DjalilNoncentral}. Improving this part in the
exchangeable case would allow to treat e.g. the adjacency matrices of uniform
random directed graphs $D(n,m_n)$. More precisely, consider a sequence $m_n$
of positive integers.
%such that $m_n/n\to\infty$, $\limsup_{n\to \infty}m_n/n^2 <1$.
Let $D_n$ be the adjacency matrix of the random graph $D(n,m_n)$, i.e.\ a
graph chosen uniformly from the set of all directed random graphs on $n$
vertices, with $m_n$ directed edges (including possible loops). One may ask
about the limiting spectral behavior of $D_n$ (depending on the behavior of
$m_n$ as $n\to\infty$). In some sense, this random graph model lies between
the oriented Erd\H{o}s-R\'enyi random graph model \cite{BCCSparse}, and the
uniform random oriented regular graph model \cite{BCAround}. The singular
values of $D_n$ (as well as its shifts by multiples of identity) may be
analyzed using Theorem \ref{th:shifted_matrix} below. More generally, one may
think about a Boltzmann-Gibbs probability distribution on the set of oriented
graphs with $n$ vertices, with density with respect to the counting measure
proportional to $\exp(-\beta_nH_n)$ where $H_n$ is a symmetric functional, and
$\beta_n>0$.

\subsection*{Outline}
In Section \ref{se:nuz} we analyze the limiting distribution of the singular
values of shifts by $\sqrt{n}z\id_n$ for the model \eqref{eq:model}, which is essential for the proof
of Theorem \ref{th:circular_law_2}, and which may be actually of independent
interest. Section \ref{se:Talagrand} is devoted to the derivation of a
combinatorial version of the Talagrand concentration inequality for convex
functions and product measures, which may be also of independent interest.
This inequality is an important tool in the proof of Theorem \ref{th:ssv} and Theorem
\ref{th:circular_law_2}. Theorems \ref{th:ssv}, \ref{th:circular_law},
\ref{th:circular_law_2} are proved in Sections \ref{se:ssv},
\ref{se:circular_law}, \ref{se:circular_law_2} respectively.

\subsection*{Notations.}
We will frequently write $K,\bfx,X$ instead of $K^{(n)},\bfx^{(n)},X^{(n)}$
for convenience, when no confusion is possible. In the whole article by $C,c$
we denote universal constants and by $C_a,c_a$ constants depending only on the
parameter $a$. In both cases the values of constants may differ between
occurrences.

\section{Limiting distribution of the singular values of shifts}
\label{se:nuz}

This section is devoted to the proof of the following result.

\begin{theorem}[Limiting singular values distribution of shifted
  matrices]\label{th:shifted_matrix}
  For every $z\in\mathbb{C}$, there exists a unique probability measure
  $\nu_z$ on $\mathbb{R}_+$, absolutely continuous with respect to the
  Lebesgue measure, depending only on $z$, such that if $X^{(n)}$ is as in
  \eqref{eq:model} with $K_n=o(n^{1/2})$, and if
  $A^{(n)}=\frac{1}{\sqrt{n}}X^{(n)}$, then, for every $z\in\mathbb{C}$,
  \begin{itemize}
  \item[(j)]
    \[
    \nu_{\sqrt{(A^{(n)}-z\id_n)^*(A^{(n)}-z\id_n)}}
    =\frac{1}{n}\sum_{k=1}^n\delta_{s_k(A^{(n)}-z\id_n)}
    \underset{n\to\infty}{\weak}\nu_z;
    \]
   \item[(jj)]
    \[
    U(z)=-\int_{\mathbb{R}_+}\!\log(s)\,\nu_z(ds)
    =
    \begin{cases}
      -\log|z| & \text{if $|z|>1$;} \\
      \frac{1-|z|^2}{2} & \text{otherwise.}
    \end{cases}
    \]
  \end{itemize}
  Furthermore, for $z=0$, the probability distribution $\nu_0$ is the
  so-called Marchenko-Pastur quarter circular law of Lebesgue density
  $x\mapsto\frac{1}{\pi}\sqrt{4-x^2}\mathbf{1}_{[0,2]}(x)$.
\end{theorem}

Theorem \ref{th:shifted_matrix} is not valid if we only assume that
$K_n=\mathcal{O}(n^{1/2})$. Indeed, one can construct $\bfx^{(n)}$ in
\eqref{eq:model} in such a way that $K_n=\mathcal{O}(n^{1/2})$ while $0$ is an
eigenvalue of $X^{(n)}$ of multiplicity at least $cn$, which would contradict
the fact that $\nu_z$ is absolutely continuous.

\begin{proof}[Proof of Theorem \ref{th:shifted_matrix}]
  First of all, note that since $K_n=o(n^{1/2})$, we have
\begin{align}\label{eq:bounds_on_moments}
  \E |X_{11}^{(n)}|^4 = o(n)
  \quad\text{and}\quad
  \E |X_{11}^{(n)}|^3 = o(n^{1/2}).
\end{align}

  If $K_n=\mathcal{O}(1)$, the result follows from \cite[Theorem 2.11 and
  Proposition 2.12]{InfoNoise}. In fact one can check that the proof given
  there works for $K_n$ being a small power of $n$.

  However, to get larger range of $K_n$ we will adapt the proof by Chatterjee
  \cite{ChatterjeeLindeberg}, who obtained the semi-circular law for symmetric
  random matrices with exchangeable entries, by using the Cauchy-Stieltjes
  trace-resolvent transform, and a generalization of the Lindeberg principle
  for the Central Limit Theorem.

  For a probability measure $\mu$ supported on $\R$, we denote by $m_\mu$ its
  Cauchy-Stieltjes transform, defined for $\xi \in \C_+=\{z\in\C:\Im(z)>0\}$ as
  (see Theorem 2.4.4 in \cite{AGZ})
  \begin{displaymath}
    m_\mu(\xi) = \int_\R \frac{1}{\lambda - \xi} \mu(d\lambda).
  \end{displaymath}
  If $M$ is a $n\times n$ matrix with real spectrum and $\xi\in\C_+$, then
  $m_{\nu_M}(\xi)$ is the normalized trace of the resolvent
  $G=(M-\xi\id_n)^{-1}$ of $M$ at point $\xi$, namely
  \[
  m_{\nu_M}(\xi)
  =\frac{1}{n}\sum_{k=1}^n(\lambda_k(M)-\xi)^{-1}
  =\frac{1}{n}\tr(G).
  \]

  Fix $z \in \C$. We start by the usual linearization trick, which consists in
  constructing a Hermitian matrix, depending linearly on $A^{(n)}-z\id_n$, and
  for which the eigenvalues are the singular values of $A^{(n)}-z\id_n$.
  Namely, we consider the $2n \times 2n$ Hermitian matrix
  \[
  B_n = \begin{pmatrix}
    0 & A^{(n)} - z\id_n\\
    (A^{(n)} - z\id_n)^* & 0
  \end{pmatrix}.
  \]
  The eigenvalues of $B_n$ are the singular values of $A^{(n)} - z\id$
  multiplied with $\pm 1$. More precisely, if $\lambda > 0$ is a singular
  value of $A^{(n)} - z\id_n$ of multiplicity $m$, then $\pm \lambda$ are both
  eigenvalues of $B_n$ of multiplicity $m$, and if zero is a singular value of
  $A^{(n)} - z\id_n$ of multiplicity $m$ then it is also an eigenvalue of
  $B_n$ of multiplicity $2m$. Thus, to prove the proposition it is enough to
  show that the spectral measure of $B_n$ converges weakly in probability to
  the same measure as the spectral measure of the matrix
  \begin{displaymath}
    \tilde{D}_n = \left(\begin{array}{cc}
        0 & \frac{1}{\sqrt{n}}G_n - z\id_n\\
        (\frac{1}{\sqrt{n}}G_n - z\id_n)^* & 0
      \end{array}\right),
  \end{displaymath}
  where $G_n=(g_{ij})_{1\le i,j\le n}$ is an $n\times n$ matrix with i.i.d.
  $\mathcal{N}(0,1)$ coefficients, as it is known \cite{BS,BCAround} that
  $\nu_{\sqrt{(n^{-1/2}G_n - z\id_n)^\ast(n^{-1/2}G_n - z\id_n)}}$ converges to the measure $\nu_z$ satisfying (jj).

  Let $b = n^{-2}\sum_{1\le i,j\le n} g_{ij}$ and let $\Ind{n}\otimes\Ind{n}$
  be the $n\times n$ matrix with all elements equal to one. Define now the
  matrix
  \[
  D_n = \begin{pmatrix}
    0 & \frac{1}{\sqrt{n}}G_n - b\Ind{n}\otimes\Ind{n}- z\id_n \\
    (\frac{1}{\sqrt{n}}G_n - b\Ind{n}\otimes\Ind{n}- z\id_n)^* & 0
  \end{pmatrix}.
  \]
  Since $D_n$ is a rank two additive perturbation of $\tilde{D}_n$, it is well
  known by using interlacing inequalities \cite{lamabook} that the Kolmogorov
  distance between $\nu_{\tilde{D}_n}$ and $\nu_{D_n}$ is at most $2/n$. Thus
  it is enough to prove that $\nu_{B_n} - \nu_{D_n} \to 0$ weakly in
  probability, which will follow if we prove that for each $\xi \in \C_+$,
  \begin{displaymath}
    m_{\nu_{B_n}}(\xi)  - m_{\nu_{D_n}}(\xi) \to 0
  \end{displaymath}
  in probability as $n \to \infty$. Since $m_{\nu_{D_n}}(\xi)$ converges in
  probability to the deterministic quantity $m_{\nu_z}(\xi)$, it is enough to
  show that for every smooth function $g$ with compact support we have
  \begin{displaymath}
    |\E g(\Re m_{\nu_{B_n}}(\xi))  - \E g(\Re m_{\nu_{D_n}}(\xi))| \to 0
  \end{displaymath}
  and that an analogous statement holds for the imaginary parts of
  $m_{\nu_{B_n}}$ and $m_{\nu_{D_n}}$.

  To this end we will use the following Theorem, which is \cite[Theorem
  1.2]{ChatterjeeLindeberg}.

  \begin{theorem}[Chatterjee-Lindeberg principle]\label{th:Chatterjee}
    Suppose $X$ is a random vector in $\R^n$ with exchangeable components of finite fourth moments and
    let $Z$ be a standard Gaussian vector in $\R^n$, independent of $X$. Define
    \[
    \hat{\mu} = \frac{1}{n}\sum_{i=1}^n X_i,
    \quad
    \hat{\sigma}^2 = \frac{1}{n}\sum_{i=1}^n (X_i-\hat{\mu})^2,
    \quad
    \bar{Z} = \frac{1}{n}\sum_{i=1}^n {Z_i},
    \quad
    Y_i = \hat{\mu} + \hat{\sigma}(Z_i-\bar{Z}).
    \]
    Let $f\colon \R^n \to \R$ be a $\mathcal{C}^3$ function, and let $L_r'(f)$
    be a uniform bound on all $r$-th partial derivatives of $f$, including
    mixed partials. For each $p$, let $w_p = \E|X_1-\hat{\mu}|^p$. Then
    \begin{displaymath}
      |\E f(X) - \E f(Y)| \le 9.5w_4^{1/2}L_2'(f)n^{1/2} + 13w_3L_3'(f)n.
    \end{displaymath}
  \end{theorem}

  It is easy to check that for any matrix valued differentiable function $M$
  of $x \in \R$ with values in the space of Hermitian matrices and $G(x) =
  (M(x) - \xi\id)^{-1}$ ($\xi \in \C_+$), we have
  \begin{align}\label{eq:matrix_differentiation}
    \frac{d}{dx} G = -G\frac{dM}{dx}G.
  \end{align}

  Consider now for $x \in \R^{n\times n}$ the $2n\times 2n$ matrix $M(x) =
  (M_{ij}(x))_{1\le i,j\le 2n}$, where $M_{ij}(x) = 0 $ if $i,j \le n$ or $i,j
  > n$, $M_{ij}(x) = \frac{1}{\sqrt{n}}x_{i,j-n} - z\ind{i=j}$ for $i \le n, j>
  n$ and $M_{ij}(x) = \frac{1}{\sqrt{n}}x_{i-n,j} - \bar{z}\ind{i=j}$ for $i > n,
  j\le n$.

  As in \cite{ChatterjeeLindeberg} we will write $\partial_{\alpha}$ for
  $\partial/\partial x_{\alpha}$ for $\alpha = (i,j)$, $1\le i,j \le n$. From
  \eqref{eq:matrix_differentiation} we get
  \begin{displaymath}
    \partial_{\alpha} \frac{1}{2n}\tr G = -\frac{1}{2n}\tr \Big(G(\partial_{\alpha} M)G\Big).
  \end{displaymath}
  Differentiating two more times and using again
  \eqref{eq:matrix_differentiation} together with the observation that for any
  $\alpha_1,\alpha_2$, $\partial_{\alpha_1}\partial_{\alpha_2} M = 0$, we get
  \begin{align*}
    \partial_{\alpha_1}\partial_{\alpha_2} \frac{1}{2n}\tr G &= \frac{1}{2n}\sum_{1\le i\neq j\le 2}\tr \Big(G(\partial_{\alpha_i} M)G(\partial_{\alpha_j}M)G\Big),\\
    \partial_{\alpha_1}\partial_{\alpha_2}\partial_{\alpha_3}\frac{1}{2n}\tr G &= - \frac{1}{2n}\sum_{1\le i\neq j\neq k \le 3}\tr \Big(G(\partial_{\alpha_i} M)G(\partial_{\alpha_j}M)G(\partial_{\alpha_k} M)G\Big).
  \end{align*}

  Note that $\partial_{\alpha} M$ has two nonzero entries, each equal to
  $n^{-1/2}$. Moreover $\|G\| \le (\Im \xi)^{-1}$. Thus
  \begin{displaymath}
    \Big|\partial_{\alpha} \frac{1}{2n}\tr G\Big| = \Big|\frac{1}{2n}\tr \Big((\partial_{\alpha} M)G^2\Big)\Big| \le \frac{1}{n^{3/2}}(\Im \xi)^{-2}.
  \end{displaymath}

  As for higher order derivatives, using the inequalities $|\tr AB |\le
  \|A\|_{\mathrm{HS}}\|B\|_{\mathrm{HS}}$ and $\|AB\|_{\mathrm{HS}} \le \|A\|
  \|B\|_{\mathrm{HS}}$, we get
  \begin{align*}
    |\partial_{\alpha_1}\partial_{\alpha_2}\frac{1}{2n}\tr G|&\le \frac{1}{2n}\sum_{1\le i\neq j\le 2}\|G\partial_{\alpha_i} M\|_{\mathrm{HS}}\|G(\partial_{\alpha_j}M)G\|_{\mathrm{HS}}\\
    & \le \frac{1}{2n}\sum_{1\le i\neq j\le 2}\|G\|^3\|\partial_{\alpha_i}M\|_{\mathrm{HS}}\|\partial_{\alpha_j} M\|_{\mathrm{HS}}\\
    &\le \frac{2}{n^{2}}(\Im \xi)^{-3}
  \end{align*}
  and similarly
  \begin{align*}
    |\partial_{\alpha_1}\partial_{\alpha_2}\partial_{\alpha_3} \frac{1}{2n}\tr G| \le \frac{6\sqrt{2}}{n}\|G\|^4 n^{-3/2} \le 9n^{-5/2}(\Im \xi)^{-4}.
  \end{align*}

  Moreover, since $x_{ij}$ are real, we have $\partial _\alpha \Re G = \Re
  (\partial_\alpha G)$ and a similar equality for the imaginary parts.
  Using this together with the above estimates on the derivatives, we get that
  for any smooth function $g$ on $\R$, with compact support, the function $f\colon \R^{n^2}\to \R$, defined as $f(x)
  = g(\frac{1}{2n}\Re \tr G(x))$, satisfies
  \begin{displaymath}
    L_2'(f) \le C_{g,\xi}n^{-2}, L_3'(f) \le C_{g,\xi}n^{-5/2}.
  \end{displaymath}

  Thus, by Theorem \ref{th:Chatterjee} (applied with $n^2$ instead of $n$) and \eqref{eq:bounds_on_moments},
  \begin{displaymath}
    |\E g(\Re m_{\nu_{B_n}}(\xi)) - \E g(\Re m_{\nu_{D_n}}(\xi))| \le C_{g,\xi}(n^{-1}(\E |X_{11}^{(n)}|^4)^{1/2} + n^{-1/2}\E |X_{11}^{(n)}|^3)\to 0
  \end{displaymath}
  as $n \to \infty$. Since an analogous convergence holds for the
  imaginary parts, this ends the proof of Theorem \ref{th:shifted_matrix}.
\end{proof}

\begin{remark}[Assumptions] One can see from the above proof that the theorem
  remains true under an assumption weaker than $K_n = o(n^{1/2})$, since it
  is enough to assume \eqref{eq:bounds_on_moments}.
\end{remark}

\section{Combinatorial Talagrand's concentration inequality}
\label{se:Talagrand}

The following result is a combinatorial analogue of the Talagrand
concentration of measure inequality for convex functions under product
measures. It plays a crucial role in our approach, and may be of independent
interest. We give a proof for completeness.

\begin{theorem}[Concentration inequality for convex
  functions]\label{th:convex_concentration}
  Let $x_1,\ldots,x_n \in [0,1]$ and let $\varphi\colon [0,1]^n \to \R$ be an
  $L$-Lipschitz convex function. Let $\pi$ be a random uniform permutation of
  the set $[n]$ and let $Z = \varphi(x_{\pi(1)},\ldots,x_{\pi(n)})$. Then for
  all $t>0$,
\begin{align}\label{eq:convex_concentration}
\p(|Z-\E Z| \ge t) \le 2\exp(-ct^2/L^2).
\end{align}
\end{theorem}

\begin{proof}[Proof of Theorem \ref{th:convex_concentration}]
  The randomness of $Z$ comes entirely from $\pi$. Let $P_n$ be the uniform
  measure on the symmetric group $S_n$. The proof follows closely the original
  argument given by Talagrand in the case of product measures
  \cite{TalNewLook} and is based on the following theorem.
  \begin{theorem}[Talagrand \cite{TalConcMeasure}]\label{th:convex_distance}
    For $\pi\in S_n$ and $A\subseteq S_n$ define
    \begin{displaymath}
      U_A(\pi) = %
      \{s\in\{0,1\}^n\colon %
      \exists_{\tau\in A} \forall_{l\le n} s_l = 0 \implies \tau(l) = \pi(l)\}.
    \end{displaymath}
    Let $V_A(\pi) = \mathrm{ConvexHull}(U_A(\pi))$ and $f(A,\pi) =
    \inf\{|s|^2\colon s\in V_A(\pi)\}$. Then
    \begin{displaymath}
      \int_{S_n} \exp\Big(\frac{1}{16}f(A,\pi)\Big)d P_n(\pi) \le \frac{1}{P_n(A)}.
    \end{displaymath}
  \end{theorem}

  In what follows, for $\tau \in S_n$ and $x \in \R^n$, we will denote $x_\tau
  = (x_{\tau(1)},\ldots, x_{\tau(n)})$. It is well known that up to universal
  constants it is enough to prove \eqref{eq:convex_concentration} with the
  median instead of the mean \cite{LedouxConcBook}. Let $Z(\tau) =
  \varphi(x_\tau)$ and $A = \{\tau\in S_n\colon Z(\tau) \le \Med Z\}$. We have
  $P_n(A) \ge 1/2$, so by Theorem \ref{th:convex_distance} and Chebyshev's
  inequality we get
    \begin{displaymath}
      P_n(\pi\colon f(A,\pi) \ge t^2) \le 2\exp(-ct^2).
    \end{displaymath}
    If $f(A,\pi) < t^2$, then there exist $s^{1},\ldots,s^{m} \in U_A(\pi)$, $p_1,\ldots,p_m \ge 0$, $p_1+\cdots+p_m = 1$, such that
    \begin{displaymath}
      \Big|\sum_{i=1}^m s^ip_i\Big| < t.
    \end{displaymath}
    Let $\tau_i \in A$ be such that $s^i_j = 0 \implies \tau_i(j) = \pi(j)$. We have for $j=1,\ldots,n$,
    \begin{align*}
      \Big|x_{\pi(j)} - \sum_{i=1}^m x_{\tau_i(j)} p_i\Big|\le \sum_{i=1}^m|x_{\pi(j)} - x_{\tau_i(j)}|p_i \le 2\sum_{i=1}^m s^i_j p_i,
    \end{align*}
    so
    \begin{displaymath}
      \Big|x_\pi - \sum_{i=1}^m p_i x_{\tau_i}\Big| \le 2\Big|\sum_{i=1}^m s^ip_i\Big| < 2t.
    \end{displaymath}
    Moreover, by convexity $\varphi(\sum_i p_i x_{\tau_i}) \le \sum_{i=1}^m
    p_i\varphi(x_{\tau_i}) \le \Med Z$ and so by the Lipschitz condition
    $\varphi(x_\pi) <\Med Z + 2Lt$. This shows that
    \[
    P_n(\pi\colon \varphi(x_\pi) \ge \Med Z + 2Lt) \le P_n(\pi\colon f(A,\pi) \ge
    t^2) \le 2\exp(-ct^2).
    \]
    To prove the bound on the lower tail let us now denote
    \[
    A = \{\tau\in S_n\colon Z(\tau) \le \Med Z-2Lt\}.
    \]
    By Theorem \ref{th:convex_distance} we have
    \begin{displaymath}
      \int_{S_n} \exp\Big(cf(A,\pi)\Big) d P_n(\pi) \le \frac{1}{P_n(A)}
    \end{displaymath}
    But if $f(A,\pi) < t^2$, then by a similar argument as above
    $\varphi(x_\pi) < \Med Z$, so the left hand side above is bounded from
    below by $\frac{1}{2}\exp(ct^2)$, which gives $P_n(A) \le 2\exp(-ct^2)$.
\end{proof}

Using integration by parts and the triangle inequality in $L_p$, we obtain the
following.

\begin{corollary}[Moments]\label{co:convex_moments} In the setting of
  Theorem \ref{th:convex_concentration} we have for any $p\ge 2$,
  \begin{displaymath}
    \|Z\|_p \le \|Z\|_1 + CL\sqrt{p}.
  \end{displaymath}
\end{corollary}

\section{Proof of Theorem \ref{th:ssv}\label{se:ssv} (smallest
  singular value)}

Throughout this section, $X^{(n)}$ is as in \eqref{eq:model}. We will prove
the following more precise version of Theorem \ref{th:ssv}, which is actually
used in the proof of Theorem \ref{th:circular_law}.

\begin{theorem}\label{th:ssv_detailed}
  Under the assumptions of Theorem \ref{th:ssv}, for any $\varepsilon \in
  (0,1)$,
  \begin{multline*}
    \p(s_n(X^{(n)}-\sqrt{n}z\id) \le \frac{1}{K_n+|z|}\varepsilon n^{-1/2}) \\
    \le CK_n^2(K_n+|z|)\log(1+K_n+|z|)\varepsilon \\
    + C \frac{K_n^4(K_n+|z|)\log^{3/2}(1+K_n+|z|)}{n^{1/2}}.
\end{multline*}
\end{theorem}

Let us first introduce some tools we will need in the proof of Theorem
\ref{th:ssv_detailed}. One of our main tools is a version of Talagrand's
concentration inequality for convex functions stated in Theorem
\ref{th:convex_concentration}, and Corollary \ref{co:convex_moments}. The next
two lemmas are rather standard corollaries. We denote by
$S_\C^{n-1}=\{z\in\C^n:|z|=1\}$ the unit Euclidean ball in $\C^n$.

\begin{lemma}[Operator norm]\label{le:operator_norm}
  For all $t > 0$ with probability at least $1 - \exp(-ct^2/K_n^2)$,
  \begin{displaymath}
    \|X^{(n)}\| \le CK_n\sqrt{n} + t.
  \end{displaymath}
\end{lemma}

\begin{proof}[Proof of Lemma \ref{le:operator_norm}]
  Let us consider a $(1/4)$-net $\mathcal{N}$ in $S_\C^{n-1}$ of cardinality
  $9^{2n}$ (it exists by standard volumetric estimates). We have
  \begin{displaymath}
    \|X^{(n)}\| \le \frac{16}{7}\sup_{x,y\in\mathcal{N}} |\langle X^{(n)}x,y\rangle|.
  \end{displaymath}
  Note that $\E \langle X^{(n)}x,y\rangle = 0$. Since for $x,y\in S_\C^{n-1}$, $A \mapsto |\langle
  Ax,y\rangle|$ is a 1-Lipschitz function with respect to the Hilbert-Schmidt
  norm, we get by Theorem \ref{th:convex_concentration} and the union bound
  \begin{displaymath}
    \forall t\geq CK_n,\quad
    \p(\|X^{(n)}\| \ge  t\sqrt{n})
    \le  2\cdot 9^{4n} \exp(-ct^2 n/K_n^2) \le \exp(-c t^2 n/K_n^2).
  \end{displaymath}
  Integrating the above inequality by parts, we get $\E \|X^{(n)}\| \le
  CK_n\sqrt{n}$. The lemma now follows by another application of Theorem
  \ref{th:convex_concentration}, this time to the function $X\mapsto \|X\|$.
\end{proof}

Let $\SPARSE(\delta)=\{x\in\C^n:\sum_{k=1}^n\mathbf{1}_{x_k\neq0}\leq\delta
n\}$ be the set of $\delta n$ sparse vectors in $\C^{n}$. Similarly as in \cite{RV} we partition the
unit sphere $S_\C^{n-1}$ into compressible vectors and incompressible vectors:
\[
\COMP(\delta,\rho) %
= \{x\in S_\C^{n-1}\colon \DIST(x,\SPARSE(\delta)) \le \rho\} %
\quad\text{and}\quad \INCOMP(\delta,\varepsilon) %
= S_\C^{n-1} \setminus \COMP(\delta,\varepsilon).
\]

If $M={(M_{ij})}_{1\leq i\leq n,1\leq j\leq k}$ and $I\subseteq [n]$, then we denote by
$M_{|I}$ the submatrix ${(M_{ij})}_{i\in I,1\leq j\leq k}$ of $M$ formed by
the rows of $M$ labeled by $I$.

\begin{lemma}[Compressible vectors]\label{le:compressible}
  For all $\alpha \in (0,1]$, for all $I \subseteq [n]$, such that
  $\#I \ge \alpha n$, for every $z \in \C$, if $n \ge C_{\alpha}K_n^2$, then
  \begin{itemize}
  \item[a)] for any $x \in S_{\C}^{n-1}$, with probability at least $1 -
    \exp(c_{\alpha} n/K_n^2)$,
    \begin{displaymath}
      |X^{(n)}_{|I} x -\sqrt{n}zx_{|I}| \ge c_{\alpha} \sqrt{n},
    \end{displaymath}
  \item[b)] with probability at least
    $1 - \exp(-c_{\alpha} n/K_n^2)$, for all $x \in \COMP(\delta,\rho)$,
    \begin{align*}
      |X^{(n)}_{|I} x -\sqrt{n}zx_{|I}| \ge c_{\alpha} \sqrt{n},
    \end{align*}
    where
    \begin{align}\label{eq:delta_and_rho}
      \rho = \frac{c_\alpha'}{K_n+|z|}
      \quad\text{and}\quad
      \delta = \frac{c_\alpha''}{K_n^2\log(1+K_n+|z|)}.
    \end{align}
  \end{itemize}
  \end{lemma}

\begin{proof}[Proof of Lemma \ref{le:compressible}]
  We abridge $K_n$ and $X^{(n)}$ into $X$ and $K$. Consider first an arbitrary
  $x \in S_\C^{n-1}$. Using \eqref{eq:xmeancov} and $|x|_1 := |x_1|+\ldots+|x_n| \le \sqrt{n}$, we
  get
  \begin{align*}
    \E |X_{|I} x - \sqrt{n}zx_{|I}|^2 %
    &= \sum_{i\in I} \E\Big|\sum_{j=1}^n X_{ij}x_j-\sqrt{n}zx_i\Big|^2\\
    & \ge \sum_{i\in I} \E \Big|\sum_{j=1}^n X_{ij}x_j\Big|^2\\
    &= \sum_{i\in I}\sum_{j=1}^n\E X_{ij}^2|x_j|^2 %
    +\sum_{i\in I}\sum_{1\le j\neq k\le n} \E X_{ij}X_{ik}x_j\bar{x}_k \\
    &\ge\# I -\frac{\#I}{n^2-1}|x|_1^2 \ge \#I/3.
  \end{align*}
  Since the function $X\mapsto |X_{|I}x-zx_{|I}|$ is convex and $1$-Lipschitz
  with respect to the Hilbert-Schmidt norm, by Corollary
  \ref{co:convex_moments} and the assumption on $I$ we get, for $n \ge
  C_{\alpha}K^2$,
  \begin{displaymath}
    \E|X_{|I} x - \sqrt{n}zx_{|I}| %
    \ge c\sqrt{\alpha n} - C K \ge c_\alpha \sqrt{n}.
  \end{displaymath}
  By Theorem \ref{th:convex_concentration} we get
  \begin{align}\label{eq:individual_vector_lower_bound}
    \p(|X_{|I} x - \sqrt{n}zx_{|I}| \le c_{\alpha}\sqrt{n} \Big) \le
    2\exp(-c_{\alpha} n/K^2),
  \end{align}
  which proves the first part of the lemma. Now, for each $\delta,\rho \in
  (0,1]$, the set $\SPARSE(\delta)\cap S_\C^{n-1}$ admits a $\rho$-net of
  cardinality
  \begin{displaymath}
    \binom{n}{\lfloor \delta n\rfloor }(3/\rho)^{2\delta n} \le \Big(\frac{C}{\rho^2 \delta}\Big)^{\delta n}.
  \end{displaymath}
  Thus for $\rho = \frac{c_\alpha'}{K+|z|}$ and $\delta =
  \frac{c_\alpha''}{K^2\log(1+K+|z|)}$ by the union bound, the inequality
  \eqref{eq:individual_vector_lower_bound}, Lemma \ref{le:operator_norm}, the
  estimate $K \ge 1$, and the triangle inequality, we get with probability at
  least
  \begin{displaymath}
    1 - \exp(-c_{\alpha} n/K^2),
  \end{displaymath}
  for all $x \in \COMP(\delta,\rho)$,
  \begin{displaymath}
    |X_{|I} x - \sqrt{n}zx_{|I}| %
    \ge c_\alpha\sqrt{n} - C\rho (CK+|z|)\sqrt{n} %
    \ge \frac{1}{2}c_\alpha \sqrt{n}.
  \end{displaymath}
\end{proof}

The final ingredient is a Berry-Esseen type estimate for linear combinations
of exchangeable random variables.

\begin{lemma}[Berry-Esseen type estimate for exchangeable
  variables]\label{le:CLT_exchangeable}
  Consider two sequences of real numbers $x = (x_1,\ldots,x_n)$ and
  $a=(a_1,\ldots,a_n)$, such that for some constants $K,L >0$,
  \begin{itemize}
  \item $\sum_{i=1}^n x_i = 0$, $\sum_{i=1}^n x_i^2 = n$, $|x_i| \le K$,
  \item $|a_i| \le \frac{L|a|}{\sqrt{n}}$, $a \neq \pm n^{-1/2}(|a|,\ldots,|a|)$.
  \end{itemize}
  Let $\pi$ be a random (uniform) permutation of $[n]$ and define
  \begin{displaymath}
    W := \sum_{i=1}^n a_i x_{\pi(i)}.
  \end{displaymath}
  Then $\sigma^2:= \E W^2 = \frac{1}{n-1}\left(n\sum_{i=1}^n a_i^2 -
    (\sum_{i=1}^n a_i)^2\right) > 0$ and
  \begin{equation}\label{eq:BE_bound}
    \sup_{t\in \R}%
    \Big|\p(W \le t) -
    \frac{1}{\sqrt{2\pi}\sigma}\int_{-\infty}^te^{-x^2/2\sigma^2}dx\Big|
    \le \frac{34LK|a|}{\sigma \sqrt{n}}.
  \end{equation}
\end{lemma}

\begin{proof}[Proof of Lemma \ref{le:CLT_exchangeable}]
  The result follows from the combinatorial Central Limit Theorem (CLT) \cite[equation (4.105) and Theorem 6.1]{MR2732624}. Namely, let
  $(c_{ij})_{i,j \le n}$ be an array of real numbers and let
  \[
  W = \sum_{i=1}^n c_{i\pi(i)},
  \]
  where $\pi$ is a random (uniform) permutation of $[n]$. Then
  \begin{displaymath}
    \sigma^2 := \Var W = \frac{1}{n-1}\sum_{i,j=1}^n \Big(c_{ij} - \frac{1}{n}\sum_{k=1}^n c_{ik} - \frac{1}{n}\sum_{k=1}^n c_{kj} + \frac{1}{n^2}\sum_{k,l=1}^n c_{kl}\Big)^2.
  \end{displaymath}
  Moreover, if $\sigma > 0$, then for every $\varepsilon > 0$,
  \begin{displaymath}
    \sup_{t \in R}%
    \left|%
      \p(\sigma^{-1}(W - \frac{1}{n}\sum_{i,j=1}^n c_{ij}) \le t)%
      - \frac{1}{\sqrt{2\pi}}\int_{-\infty}^t\exp(-x^2/2)dx)%
      \right|
      \le \frac{16.3A}{\sigma},
  \end{displaymath}
  where
  \begin{displaymath}
  A = \max_{1\le i,j\le n}\Big|c_{ij} - \frac{1}{n}\sum_{k=1}^n c_{ik} - \frac{1}{n}\sum_{k=1}^n c_{kj} + \frac{1}{n^2}\sum_{k,l=1}^n c_{kl}\Big|.
  \end{displaymath}
  The lemma follows by setting $c_{ij} = a_i x_j$ and elementary calculations.
\end{proof}

\begin{remark}[Conditional expectation] \label{re:conditional} Before we
  proceed with the proof let us make a comment concerning conditional
  expectation. Namely, in the proof of Theorem \ref{th:ssv_detailed} as well
  as in other proofs in the article we will encounter a situation in which a
  deterministic set $E \subset [n]\times [n]$ is given and in addition one
  considers a random set $F \subseteq ([n]\times [n]) \setminus E$, measurable
  with respect to $\mathcal{F}'=\sigma((\pi(i,j))_{(i,j)\in E})$. One can then
  consider the $\sigma$-field $\mathcal{F} =
  \sigma(\mathcal{F}',(\pi(i,j))_{(i,j)\in F})$ and the random vector $\pi' =
  (\pi(i,j))_{(i,j)\in (E\cup F)^c}$ (this notation is slightly informal as
  the set $F$ itself is random, but this should not lead to misunderstanding).
  One can then see that conditionally on $\mathcal{F}$, $(\pi(i,j))_{(i,j)\in
    (E\cup F)^c}$ is distributed as a random (uniform) bijection from the set
  $(E\cup F)^c$ to the set $([n]\times [n])\setminus \{\pi(i,j)\colon (i,j)\in
  E \cup F\}$. In particular the vector $(X_{ij})_{(i,j)\in (E\cup F)^c}$ is
  distributed as a random permutation of the sequence obtained from
  $(\bfx_{ij})_{i,j=1}^n$ after removing the elements $(X_{ij})_{(i,j)\in
    E\cup F}$ and one can apply to it conditionally e.g.\ Theorem
  \ref{th:convex_concentration} or Lemma \ref{le:CLT_exchangeable}.
\end{remark}

\begin{proof}[Proof of Theorem \ref{th:ssv_detailed}] Again we will write $X$ instead of $X^{(n)}$ and $K$ instead of $K_n$. We may and will assume
  that
  \begin{align}\label{eq:assumption_on_n}
    n \ge CK^2\log(1+K+|z|)
  \end{align}
  (otherwise the bounds of the theorem become trivial as the right-hand side
  exceeds one).

Let $B = X - \sqrt{n}z\id$ and $s_n = s_n(B)$. Since $s_n = \inf_{x\in S_\C^{n-1}} |Bx|$, it is enough to bound from below
the quantities
\begin{displaymath}
\beta = \inf_{x\in \COMP(\delta,\rho)} |Bx|, \quad \zeta = \inf_{x\in \INCOMP(\delta,\rho)} |Bx|
\end{displaymath}
for some $\delta = \delta_{K,z}$, $\rho = \rho_{K,z}$. By Lemma \ref{le:compressible} for
 \begin{align}\label{eq:def_rho_delta}
\rho = \frac{c}{K+|z|}\;{\rm  and}\;  \delta = \frac{c}{K^2\log(1+K+|z|)},
\end{align}
 we have for $n \ge CK^2$,
\begin{align}\label{eq:bound_on_beta}
\beta \ge c\sqrt{n}
\end{align}
with probability at least $1 - \exp(-cn/K^2)$.
By Lemma 3.5 from \cite{RV},
\begin{displaymath}
\p(\zeta \le \rho \varepsilon n^{-1/2}) \le \frac{1}{\delta n}\sum_{i=1}^n \p(\DIST(X_i - z\sqrt{n}e_i,H_i)\le \varepsilon),
\end{displaymath}
where $X_i$'s are the columns of $X$ and $H_i = {\rm span}(X_k\colon k\neq i)$.

Let now $\eta^{i}$ be a random (unit) normal to $H_i$ defined as a measurable function of $(X_1, \ldots, X_{i-1}, X_{i+1}, \ldots, X_n)$.
We have $\DIST(X_i - z\sqrt{n} e_i, H_i) \ge |\langle X_i - z\sqrt{n} e_i,\eta^i\rangle|$,
so by exchangeability of the entries of $X$,
\begin{align}\label{eq:small_ball_wanted}
\p(\zeta \le \rho\varepsilon n^{-1/2}) \le \frac{1}{\delta} \p(|\langle X_n,\eta^n\rangle -z\sqrt{n}\langle e_n,\eta^n\rangle |\le \varepsilon).
\end{align}

To estimate the right hand side we will use Lemma~\ref{le:CLT_exchangeable}.
To be able to apply it, we have to prove certain properties of the vector
$\eta = \eta^n$. Let us define for $\theta,r \in [0,1]$,
\begin{displaymath}
\mathcal{G}(\theta,r) = \bigcup_{I \subseteq [n], \#I \ge (1-\theta)n} \mathcal{G}_I(\theta,r),
\end{displaymath}
where
\begin{displaymath}
\mathcal{G}_{I}(\theta,r) = \{x\in S_\C^{n-1}\colon \exists_{\xi \in \C, |\xi|\le 1} |x_{|I} - \frac{\xi}{\sqrt{\#I}}\Ind{I}|\le r\}
\end{displaymath}
(here $\Ind{I} = (u_i)_{i=1}^{\# I}$ with $u_i = 1$ for all $i$). % \in I$ and $x_i= 0$ for $i\in I^c$).
We will show that with high probability $\eta\notin \mathcal{G}(\theta, r)$
for certain $r = r_{K,z}$ and $\theta = \theta_{K,z}$.
Note that if $\#I \ge (1-\theta)n$, then $\mathcal{G}_I(\theta,r)$ admits a $3r$-net $\mathcal{N}_I$, such that
\begin{displaymath}
\#\mathcal{N}_I \le \frac{C}{r^2}\Big(\frac{3}{r}\Big)^{2\theta n}.
\end{displaymath}
Indeed, for any $x \in \mathcal{G}_I(\theta,r)$, $x_{|I}$ can be approximated up to $2r$ by one of the vectors $\frac{\xi}{\sqrt{\#I}}\Ind{I}$, where $\xi$ comes from an $r$-net in the unit disk in $\C$, whereas $x_{|I^c}$ can be approximated by a vector from an $r$-net in the unit ball of $\C^{I^c}$.
Thus the set $\mathcal{G}(\theta,r)$ admits a $3r$ net $\mathcal{N}$ of cardinality
\begin{displaymath}
\#\mathcal{N} \le \sum_{k=0}^{\lfloor \theta n\rfloor} \binom{n}{k}\frac{C}{r^2}\Big(\frac{3}{r}\Big)^{2\theta n} \le \frac{C}{r^2}\Big(\frac{en}{\lfloor \theta n\rfloor }\Big)^{\theta n} \Big(\frac{3}{r}\Big)^{2\theta n} \le \exp\Big(Cn\theta \log\Big(\frac{2}{\theta r}\Big)\Big)
\end{displaymath}
for $n\ge 1/\theta$. Let $B'' = (B')^\ast$, where $B'$ is the matrix obtained by removing the last column of $B$. By exchangeability, $B''$ has the same distribution as $(X-\bar{z}\sqrt{n}\id)_{|\{1,\ldots,n-1\}}$ and so by the first part of Lemma \ref{le:compressible}, the union bound and Lemma \ref{le:operator_norm} we get with probability at least $1 - \exp(-cn/K^2)$,
\begin{displaymath}
\inf_{x\in \mathcal{G}(\theta, r)}|B'' x| \ge c\sqrt{n}
\end{displaymath}
for
\begin{align}\label{eq:r_and_theta}
r = \frac{c}{K+|z|}, \theta = \frac{c}{K^2\log(1+K+|z|)}
 \end{align}
(note that \eqref{eq:assumption_on_n} implies that $n \ge 1/\theta$). We may and will assume that $\theta \le 1/2$. Combining the above estimate with the equality $B''\eta = 0$, we get that with probability at least $1 - \exp(-c n/K^2)$,
\begin{align}\label{eq:not_in_G}
\eta \notin \mathcal{G}(\theta,r).
\end{align}

Now let $I = \{i\in [n]\colon |\eta_i|\le \frac{1}{\sqrt{\theta n}}\}$ and
note that by Markov's inequality $\# I \ge n (1-\theta)$. Let $\mathcal{F}'$
be the $\sigma$-field generated by $(\pi(i,j))_{1\le i\le n,1\le j\le n-1}$.
In particular $\eta$ is $\mathcal{F}'$-measurable. Let also $\mathcal{F} =
\sigma(\mathcal{F}',\{\pi(i,n)\}_{i\notin I})$.

Note that $\sum_{i\in I} X_{in}$ and $\sum_{i\in I} X_{in}^2$ are $\mathcal{F}$-measurable. We will use this fact together with Remark \ref{re:conditional} to estimate the conditional distribution
\begin{displaymath}
\p(|\langle X_n,\eta\rangle -Y |\le \varepsilon|\mathcal{F}),
\end{displaymath}
where $Y$ is any $\mathcal{F}$-measurable random variable, by another random
variable measurable with respect to $\mathcal{F}$. To this end we will use
Lemma \ref{le:CLT_exchangeable}.

Let us define $\hat{\mu} = \frac{1}{\# I}\sum_{i\in I} X_{in}$, $\hat{\Sigma}^2 = \frac{1}{\# I}\sum_{i\in I}(X_{in} - \hat{\mu})^2$.

Let $\gamma$ be a small constant which will be fixed later on. Note that by
Theorem \ref{th:convex_concentration} we have
\begin{align}\label{eq:conditional_1}
\p(|\sum_{i=1}^n X_{in}| \ge \gamma n) \le 2\exp(- c\gamma^2 n/K^2).
\end{align}
Now, by exchangeability of $X_{in}, i =1,\ldots,n$ with respect to $\p(\cdot|\mathcal{F'})$, we get that on the set $\Delta_1 = \{|\sum_{i=1}^n X_{in}| \le \gamma n \} \in \mathcal{F'}$,
\begin{displaymath}
|\E(\sum_{i\in I} X_{in}|\mathcal{F'})| = |\frac{\#I}{n}\sum_{i=1}^n X_{in}| \le \gamma n.
\end{displaymath}
Thus, again by Theorem \ref{th:convex_concentration} we get on $\Delta_1$,
\begin{align}\label{eq:conditional_2}
\p(|\sum_{i\in I} X_{in}| \ge 2\gamma n|\mathcal{F'}) \le 2\exp(-c \gamma^2 n^2 /(K^2\#I )) \le 2\exp(-c \gamma^2 n/K^2).
\end{align}
Let now $\Delta_2 = \{|\hat{\mu}| \le 4\gamma \} \in \mathcal{F}$. Combining \eqref{eq:conditional_1}, \eqref{eq:conditional_2} and the inequality $\# I \ge (1-\theta)n \ge n/2$, we get
\begin{align}\label{eq:bound_on_Delta_2}
\p(\Delta_2) \ge 1 - 4\exp(-c \gamma^2 n/K^2).
\end{align}
Now, by Corollary \ref{co:convex_moments} applied to $Z = \sqrt{\sum_{i=1}^n
  X_{in}^2}$, for $n \ge 4C^2K^2$ we get
\begin{displaymath}
\sqrt{n} = \sqrt{\E\sum_{i=1}^n X_{in}^2}\ge \E\sqrt{\sum_{i=1}^n X_{in}^2}\ge \sqrt{n} - CK \ge \frac{1}{2}\sqrt{n}
\end{displaymath}
and by Theorem \ref{th:convex_concentration} the set $\Delta_3 = \{Cn\ge
\sum_{i=1}^n X_{in}^2 \ge cn\}\in \mathcal{F}'$ satisfies
\begin{align}\label{eq:bound_on_Delta_3}
\p(\Delta_3) \ge 1 - 2\exp(-c n/K^2).
\end{align}
On $\Delta_3$, $\E(\sum_{i\in I} X_{in}^2 |\mathcal{F}') \ge c(1-\theta) n$
and by Corollary \ref{co:convex_moments},
\begin{displaymath}
\E(\sqrt{\sum_{i\in I} X_{in}^2 }|\mathcal{F}') \ge \sqrt{c(1-\theta)n} - CK \ge c\sqrt{n},
\end{displaymath}
for $n \ge C'K^2$, so yet another application of Theorem
\ref{th:convex_concentration} yields that on $\Delta_3$,
\begin{displaymath}
\p(\sum_{i\in I} X_{in}^2 \le c n|\mathcal{F}') \le 2\exp(-c n/K^2).
\end{displaymath}
Combining this with \eqref{eq:bound_on_Delta_3} we get
\begin{align}\label{eq:sum_of_squares}
\p(Cn\ge \sum_{i\in I} X_{in}^2 \ge c n) \ge 1 - 2\exp(-cn/K^2).
\end{align}
Since $\hat{\Sigma}^2 = \frac{1}{\#I}\sum_{i\in I} X_{in}^2 - \hat{\mu}^2$, by the above inequality and \eqref{eq:bound_on_Delta_2} (applied with a sufficiently small universal constant $\gamma$ adjusted to the constant $c$ in \eqref{eq:sum_of_squares}) we get
\begin{displaymath}
\p(\hat{\Sigma}^2 \ge c ) \ge 1 - 2\exp(-c n/K^2).
\end{displaymath}

Let
\begin{displaymath}
\Delta = \{\sum_{i\in I}X_{ni}^2\le Cn, \hat{\Sigma}^2 \ge c , \hat{\mu} \le 1\;\textrm{and}\; \eta \notin \mathcal{G}(\theta,r)) \in \mathcal{F}.
\end{displaymath}
 By the above inequality together with \eqref{eq:not_in_G},\eqref{eq:bound_on_Delta_2},\eqref{eq:sum_of_squares} we have
\begin{align}\label{eq:bound_on_Delta}
\p(\Delta) \ge 1 - 2\exp(-cn/K^2).
\end{align}
On $\Delta$ we have ${\rm dist}(\eta_{|I},\{\frac{\xi}{\sqrt{\#I}} \Ind{I}\colon \xi \in \C, |\xi|\le 1\}) \ge r$. Let $\eta'$ and $\eta''$ be resp. the real and imaginary part of $\eta$. Thus at least one of the vectors $\eta'_{|I},\eta''_{|I}$ is at distance at least $r/\sqrt{2}$ from $\{\xi \Ind{I}\colon \xi \in \R\})$. Indeed,  otherwise we would have $|\eta'_{|I} - \frac{\xi'}{\sqrt{\#I}}\Ind{I}|,  |\eta''_{|I} - \frac{\xi''}{\sqrt{\#I}}\Ind{I}| \le r/\sqrt{2}$ for some $\xi',\xi'' \in \R$, and so
\begin{displaymath}
|\eta_{|I} - \frac{\xi'+i\xi''}{\sqrt{\#I}}\Ind{I}| \le r,
\end{displaymath}
which implies that ${\rm dist}(\eta_{|I},y) \le r$, where $y$ is the orthogonal projection of $\eta_{|I}$ onto ${\rm span}(\Ind{I})$. Since $|\eta_{|I}|\le 1$ we have $|y|\le 1$, so $y = \frac{\xi}{\sqrt{\# I}}\Ind{I}$ for some $\xi$ with $|\xi| \le 1$, which gives a contradiction.

We will consider the case
\begin{align}\label{eq:assumption_on_eta_prime}
{\rm dist}(\eta'_{|I},\{\xi \Ind{I}\colon \xi \in \R\}) \ge r/\sqrt{2},
\end{align}
the other one is analogous.

Assume now that for some $\mathcal{F}$-measurable complex random variable $Y = Y' + iY''$, $|\langle \eta,X_n\rangle - Y|\le \varepsilon$. Then $|\langle \eta',X_n\rangle - Y'|\le \varepsilon$. Define $x_i = (X_{in} - \hat{\mu})\hat{\Sigma}^{-1}$ for $i \in I$.

Since $\eta'$, $X_{in}, i\notin I$ and $\hat{\mu}$ are $\mathcal{F}$-measurable we have on $\Delta$,
\begin{align}\label{eq:small_ball_reduction}
\p(|\langle \eta,X_n\rangle - Y|\le \varepsilon|\mathcal{F}) \le \sup_{u \in \R} \p(|\sum_{i\in I} \eta'_i x_i - u| \le \varepsilon\hat{\Sigma}^{-1}|\mathcal{F}) \le \sup_{u \in \R} \p(|\sum_{i\in I} \eta'_i x_i - u| \le C\varepsilon|\mathcal{F}),
\end{align}
where in the second inequality we used the fact that on $\Delta$, $\hat{\Sigma}^2 \ge c$.
Moreover, by Remark \ref{re:conditional}, the right-hand side above equals
\begin{displaymath}
\sup_{u \in \R} \p(|\sum_{i\in I} \eta'_i x_{\tau(i)} - u| \le C\varepsilon|\mathcal{F}),
\end{displaymath}
where $\tau$ is a random permutation of $I$, distributed (conditionally on
$\mathcal{F})$ uniformly, and so we are in position to use Lemma
\ref{le:CLT_exchangeable}. Denote $W = \sum_{i\in I} \eta'_i x_{\tau(i)}$,
$\sigma^2 = \E (W^2|\mathcal{F})$. Note that $\sum_{i\in I} x_i = 0$ and
$\sum_{i\in I} x_i^2 = \#I$. Moreover on $\Delta$ we have $|x_i| \le CK$.
Using the fact that the density of a Gaussian distribution with variance
$\sigma^2$ is bounded from above by $\sigma^{-1}$ and $\#I \ge n/2$, we get by
Lemma \ref{le:CLT_exchangeable} that on $\Delta$,
\begin{align*}
\p(|\sum_{i\in I} \eta'_i x_{\tau(i)} - u| \le \varepsilon|\mathcal{F}) \le C\varepsilon \sigma^{-1} +
C\frac{KL|\eta'_{|I}|}{\sigma n^{1/2}},
\end{align*}
where $L = \max_{i\in I} \sqrt{\# I}|\eta_i'|/|\eta'_{|I}| \le \max_{i\in I} \sqrt{n}|\eta_i'|/|\eta'_{|I}|$. Note that by the definition of $I$, we have $|\eta_i'| \le 1/\sqrt{\theta n}$, so
$L \le \frac{1}{\sqrt{\theta}|\eta'_{|I}|}$ and the above bound implies that on $\Delta$,
\begin{align}\label{eq:conditional_BE}
\p(|\sum_{i\in I} \eta'_i x_{\tau(i)} - u| \le \varepsilon|\mathcal{F}) \le C\varepsilon \sigma^{-1} +
C\frac{K}{\sigma \theta^{1/2}n^{1/2}}.
\end{align}

It remains to estimate $\sigma$ from below on the set $\Delta$.

By Lemma \ref{le:CLT_exchangeable} we have
\begin{displaymath}
\sigma^2 = \frac{\#I}{\#I - 1} \Big(\sum_{i\in I} \eta_i'^2 -\frac{1}{\#I}(\sum_{i\in I} \eta'_i)^2\Big) \ge \sum_{i\in I} \eta_i'^2 -\frac{1}{\#I}(\sum_{i\in I} \eta'_i)^2.
\end{displaymath}
The function $t\mapsto \sqrt{t}$ is $C/|\eta'_{|I}|$ Lipschitz on $(|\eta'_{|I}|^2/2,\infty)$, so we get that if $\sum_{i\in I} (\eta'_i)^2 - \frac{1}{\#I}(\sum_{i\in I} \eta_i)^2 \le \kappa$ with $\kappa < |\eta'_{|I}|^2/4$, then
\begin{displaymath}
(\sum_{i\in I} (\eta'_i)^2)^{1/2} - \frac{1}{\sqrt{\#I}}|\sum_{i\in I} \eta_i| \le C \frac{\kappa}{|\eta'_{|I}|},
\end{displaymath}
which can be rewritten as
\begin{displaymath}
|\eta'_{|I}| - |\langle \eta'_{|I},\frac{1}{\sqrt{\#I}}\Ind{I}\rangle| \le C\frac{\kappa}{|\eta'_{|I}|}.
\end{displaymath}
Multiplying both sides by $2|\eta'_{|I}|$ and using that $|\frac{\Ind{I}}{\sqrt{\#I}}| = 1$, we get
\begin{align}\label{eq:dist-eta-to-ones}
|\eta'_{|I} - \frac{h}{\sqrt{\#I}}\Ind{I}|^2 %\le C\frac{\kappa}{\sqrt{r}} |\eta'_{|I}|
  \le C\kappa
\end{align}
where $h = |\eta'_{|I}| {\rm sgn}(\langle
\eta'_{|I},\frac{1}{\sqrt{\#I}}\Ind{I}\rangle)$. From
\eqref{eq:assumption_on_eta_prime} it follows, that $|\eta'_{|I}|^2 \ge
r^2/2$. For $\kappa = cr^{2}$ with a sufficiently small absolute constant $c$
(in particular we want to assure that $\kappa \le |\eta'_{|I}|^2/4$), the
right hand side of~\eqref{eq:dist-eta-to-ones} is smaller than $r^2/2$ and so
(again by \eqref{eq:assumption_on_eta_prime}) the inequality
\eqref{eq:dist-eta-to-ones} cannot hold on $\Delta$.

Thus, on $\Delta$ we have
\begin{displaymath}
\sigma \ge cr,
\end{displaymath}
which, when combined with \eqref{eq:conditional_BE} gives
\begin{displaymath}
\p(|\sum_{i\in I} \eta'_i x_{\tau(i)} - u| \le \varepsilon|\mathcal{F}) \le C\frac{\varepsilon}{r} +
C\frac{K}{r \theta^{1/2}n^{1/2}}
\end{displaymath}
on $\Delta$.
Going now back to \eqref{eq:small_ball_reduction}, \eqref{eq:bound_on_Delta} and \eqref{eq:small_ball_wanted} we get
\begin{displaymath}
\p(\zeta \le \rho\varepsilon n^{-1/2}) \le C\frac{\varepsilon}{\delta r} +
C\frac{K}{\delta r \theta^{1/2}n^{1/2}} + \frac{2}{\delta}\exp(-cn/K^2).
\end{displaymath}
Together with \eqref{eq:bound_on_beta} this gives (after adjusting $c$)
\begin{align*}
\p(s_n(B) \le \rho\varepsilon n^{-1/2})
\le C\frac{\varepsilon}{\delta r}
+
C\frac{K}{\delta r \theta^{1/2}n^{1/2}}
+ 2\exp(-cn/K^2).
\end{align*}
Plugging in the values of $\rho,\delta$ (equation \eqref{eq:def_rho_delta}) and $\theta, r$ (equation \eqref{eq:r_and_theta}) we get
\begin{align*}
&\p(s_n(B) \le \frac{1}{K+|z|}\varepsilon n^{-1/2}) \\
\le &  2\exp(-cn/K^2) + CK^2(K+|z|)\log(1+K+|z|)\varepsilon +
C \frac{K^4(K+|z|)\log^{3/2}(1+K+|z|)}{n^{1/2}},
\end{align*}
which ends the proof (we again adjust the constants to remove the first term
on the right-hand side).
\end{proof}

\section{Proof of Theorem \ref{th:circular_law}\label{se:circular_law} (circular law for the first model)}
\label{sec:proofs_circular}

By \cite[Lemma A2]{BCC2} or \cite{BCAround}, to prove that $\nu_{A^{(n)}}$
converges weakly in probability to the uniform measure on the unit disc, it is
enough to show that the following is true:
\begin{itemize}
\item[(i)] \textbf{Singular values of shifts.} For all $z \in \C$, there
  exists a non random probability measure $\nu_z$ on $\mathbb{R}_+$,
  absolutely continuous with respect to the Lebesgue measure, depending only
  on $z$, and such that
  \[
  \nu_{z,n} :=\frac{1}{n}\sum_{k=1}^n\delta_{s_k(A^{(n)}-z\id_n)}
  \underset{n\to\infty}{\weak}\nu_z.
  \]
  Moreover, for almost all $z\in \C$,
\begin{align*}
    U(z) := -\int_{\R_+} \log (s) \nu_z(ds) =
    \left\{
      \begin{array}{cc}
        -\log |z| & \textrm{if $|z| >1$},\\
        \frac{1}{2}(1-|z|^2) & \textrm{otherwise};
      \end{array}
    \right.
\end{align*}
\item[(ii)] \textbf{Uniform integrability.} For all $z \in \C$, the function
  $s \mapsto \log(s)$ is uniformly integrable in probability with respect to
  the family of measures $\{\nu_{z,n}\}_{n\ge 1}$, i.e.
  \begin{displaymath}
    \forall\varepsilon>0,\quad
    \lim_{t\to \infty}\limsup_{n\to \infty}
    \p\Big(\int_{\R_+} |\log s|\ind{|\log s| > t}d\nu_{z,n}(s) > \varepsilon\Big)=0.
\end{displaymath}
\end{itemize}

The first item (i) is settled by Theorem \ref{th:shifted_matrix}. It remains
to prove assertion (ii), which is the aim of the rest of this section. We
follow the Tao and Vu approach \cite{TV}. We will combine estimates on the
smallest singular value coming from Section \ref{se:ssv} with a rougher bound
on intermediate singular values of the matrix, obtained in the following
lemmas.

\begin{lemma}[Distance to a random subspace]\label{le:distance} Let $R$ be a
  deterministic $n\times n$ matrix. Denote the rows of $X^{(n)} + R$ by
  $Z_1,\ldots,Z_n$. Consider $k\le n-1$ and let $H$ be the random subspace of
  $\C^n$ spanned by $Z_1,\ldots,Z_k$. Then with probability at least $1-
  2\exp(-c(n-k)/K_n^2)$,
  \begin{displaymath}
    \DIST(Z_{k+1},H) \ge c\sqrt{n-k}.
  \end{displaymath}
\end{lemma}

\begin{proof}[Proof of Lemma \ref{le:distance}]
  In what follows we will suppress the superscript $(n)$ and write simply $X_{ij}$
  for the entries of the matrix $X^{(n)}$. The rows of $X^{(n)}$ will be denoted by $X_1,\ldots,X_n$.

  We can assume that $k \le n-CK_n^2$ for some absolute constant $C$, otherwise
  the estimate on probability given in the lemma becomes trivial for $c$
  sufficiently small. Consider the $\sigma$-field $\mathcal{F}$ generated by
  $(\pi(i,j))_{1\le i\le k,1\le j\le n}$.

  We will first replace $H$ by $\tilde{H} = \textrm{span}(H,\Ind{},Y)$ where $\Ind{}$
  is the vector of ones and $Y$ is the $(k+1)$-st row of $R$. We can assume
%  that $\tilde{H}$ is with probability one of dimension $l = k+2$ (otherwise
  that with probability one $\tilde{H}$ is of dimension $l = k+2$ (otherwise
  we may enlarge $\tilde{H}$ to a subspace of dimension $l$ in an
  $\mathcal{F}$-measurable way). Let $X = X_{k+1} - \hat{\mu}\Ind{}$, where
  $\hat{\mu} = \frac{1}{n(n-k)} \sum_{i=k+1}^n\sum_{j=1}^n X_{ij}$ (note that
  $\hat{\mu}$ is $\mathcal{F}$-measurable) and denote the coordinates of $X$
  by $x_1,\ldots,x_n$. Note that $\DIST(Z_{k+1},H) \ge \DIST(Z_{k+1},\tilde{H})= \DIST(X,\tilde{H})$.

  We have $\E(x_i|\mathcal{F}) = 0$. Moreover, since $\E \hat{\mu} = 0$, by
  Theorem \ref{th:convex_concentration} we have
  \begin{align}\label{ineq:bound_on_the_mean}
    \p(|\hat{\mu}| \ge \gamma) \le 2\exp(-c\gamma^2n(n-k)/K_n^2).
  \end{align}
  Consider now $M^2 := \sum_{i=k+1}^n \sum_{j=1}^n X_{ij}^2$ (which is
  $\mathcal{F}$-measurable) and note that by Corollary \ref{co:convex_moments}, $\E M \ge \sqrt{\E M^2} - CK_n \ge
  2^{-1}\sqrt{n(n-k)}$ (where we used the assumption $n-k \ge CK_n^2$). Thus
  again by Theorem \ref{th:convex_concentration} we have
  \begin{align}\label{ineq:bound_on_the_variance}
    \p( M^2 \le cn(n-k)) \le 2\exp(-c n(n-k)/K_n^2).
  \end{align}
  Define $\hat{\Sigma}^2 = \E(x_j^2|\mathcal{F}) =
  \frac{1}{n(n-k)}\sum_{i=k+1}^{n}\sum_{j=1}^n (X_{ij} - \hat{\mu})^2 =
  \frac{1}{n(n-k)}M^2 - \hat{\mu}^2$ and
  \begin{align*}
    \Delta = \{\hat{\Sigma} \ge c\}.
  \end{align*}
%  By the above estimates with $\gamma$ small enough, we have
  Combining~\eqref{ineq:bound_on_the_variance} and~\eqref{ineq:bound_on_the_mean} with $\gamma$ small enough we obtain
  \begin{align}\label{eq:bound_on_the_other_delta}
    \p(\Delta) \ge 1 - 2\exp(-cn(n-k)/K_n^2).
  \end{align}
  Let $P$ be the orthogonal projection on $\tilde{H}^\perp$ and let
  $e_1,\ldots,e_{n-l}$ be an orthonormal basis in $\tilde{H}^\perp$. Denote $e_i =
  (e_{ij})_{1\le j\le n}$. We have
  \begin{align*}
    \E (\DIST(X,\tilde{H})^2|\mathcal{F}) &= \E (|PX|^2|\mathcal{F}) =  \sum_{i=1}^{n-l} \E \Big(\Big|\sum_{j=1}^n x_j \bar{e}_{ij}\Big|^2|\mathcal{F}\Big)\\
    & = \sum_{i=1}^{n-l} \sum_{j=1}^n |e_{ij}|^2\E(x_j^2|\mathcal{F}) +
    \sum_{i=1}^{n-l}\sum_{1\le j\neq t \le n} \bar{e}_{ij}e_{it} \E
    (x_{j}x_{t}|\mathcal{F}).
  \end{align*}
  We have
  \begin{displaymath}
    \E(x_j^2|\mathcal{F}) = \hat{\Sigma}^2
  \end{displaymath}
  and for $j\neq t$,
  \begin{align*}
    \E(x_j x_t|\mathcal{F}) &= \frac{1}{n(n-k)[n(n-k)-1]} \sum_{\stackrel{(a,b)\neq (r,s)}{k+1 \le a,r\le n, 1\le b,s \le n}} (X_{ab}-\hat{\mu})(X_{rs}-\hat{\mu}) \\
    &= - \frac{1}{n(n-k)[n(n-k)-1]}\sum_{a = k+1}^n \sum_{b=1}^n (X_{ab} -
    \hat{\mu})^2 = - \frac{1}{n(n-k)-1}\hat{\Sigma}^2.
  \end{align*}
  Thus
  \begin{align*}
    \E (\DIST(X,\tilde{H})^2|\mathcal{F}) &= \hat{\Sigma}^2\sum_{i=1}^{n-l}\Big(\sum_{j=1}^n |e_{ij}|^2 - \frac{1}{n(n-k)-1}\sum_{1\le j\neq t\le n} \bar{e}_{ij}e_{it}\Big)\\
    &=\hat{\Sigma}^2\sum_{i=1}^{n-l}\Big(\sum_{j=1}^n |e_{ij}|^2 - \frac{1}{n(n-k)-1}\Big|\sum_{ j=1}^n \bar{e}_{ij}\Big|^2 + \frac{1}{n(n-k)-1}\sum_{j=1}^n |e_{ij}|^2\Big)\\
    &\ge \hat{\Sigma}^2 \sum_{i=1}^{n-l}\Big(\frac{n(n-k)}{n(n-k)-1} - \frac{n}{n(n-k)-1}\Big)\\
    & = \hat{\Sigma}^2 (n-l)\frac{n^2- kn - n}{n^2 - kn - 1}.
  \end{align*}
  Thus, using the assumption that $k\le n-C$ we get for $n \ge C$ that on
  $\Delta$,
  \begin{displaymath}
    \E (\DIST(X,\tilde{H})^2|\mathcal{F}) \ge c (n-k).
  \end{displaymath}
  Using now Corollary \ref{co:convex_moments}, the fact that $P$ is 1-Lipschitz
  and the assumption $n-k \ge CK_n^2$, we get on $\Delta$,
  \begin{displaymath}
    \E (\DIST(X,\tilde{H})|\mathcal{F}) \ge c\sqrt{n-k}.
  \end{displaymath}
  Applying Theorem \ref{th:convex_concentration} we obtain on $\Delta$,
  \begin{displaymath}
    \p(\DIST(X,\tilde{H}) \ge c\sqrt{n-k}|\mathcal{F}) \ge 1 -2\exp(-c(n-k)/K_n^2),
  \end{displaymath}
  which when combined with \eqref{eq:bound_on_the_other_delta} gives
  \begin{displaymath}
    \p(\DIST(X,\tilde{H}) \ge c\sqrt{n-k}) \ge 1- 2\exp(-c(n-k)/K_n^2).
  \end{displaymath}
\end{proof}

\begin{lemma}[Lower bound on the intermediate singular
  values]\label{le:intermediate_sv}
  Let $s_1 \ge \ldots\ge s_n$ be the singular values of $A^{(n)} - z\id_n$. If
  $K_n = \mathcal{O}(n^{1/(2+\delta)})$ for some $\delta > 0$, then there exists $\gamma
  \in (0,1)$ such that for every $z \in \C$ we have
  \begin{displaymath}
    \lim_{n\to\infty}\p\Big(\exists_{n^\gamma \le i \le n-1} \  s_{n-i} \le c \frac{i}{n}\Big)=0.
  \end{displaymath}
\end{lemma}

\begin{proof}[Proof of Lemma \ref{le:intermediate_sv}]
  The proof follows an argument due to Tao and Vu \cite{TV}. Let $R =
  -\sqrt{n}z\id$ and recall the notation of Lemma \ref{le:distance}. For some
  $\gamma \in (0,1)$ to be chosen later on, consider $i \ge n^\gamma$. Let $k
  = n - \lfloor i/2\rfloor$ and let $B$ be the $k\times n$ matrix with rows
  $Z_1,\ldots,Z_k$. By Cauchy interlacing inequalities we have
  $\sqrt{n}s_{n-j} = s_{n-j}(X^{(n)}+R) \ge s_{n-j}(B)$ for $j\ge \lfloor
  i/2\rfloor$. Let $H_j$, $j = 1,\ldots,k$ be the subspace of $\C^n$ spanned
  by all the rows of $B$ except for the $j$-th one. By \cite[Lemma A4]{TV},
  \[
  \sum_{j=1}^k s_j(B)^{-2} = \sum_{j=1}^k \DIST(Z_j,H_j)^{-2}.
  \]
  By Lemma \ref{le:distance}, for each $j\le k$, $\DIST(Z_j,H_j) \ge
  c\sqrt{n-k+1}\ge c\sqrt{i}$ with probability at least $1 - 2\exp(-c i/K_n^2) \ge 1 - 2\exp(-c n^\gamma/K_n^2)$. (Note that we can use the lemma here thanks to exchangeability of the rows of the matrix).
  By the union bound,
  with probability at least $1 - 2n\exp(-c n^\gamma/K_n^2)$, we get
  \[
  \sum_{j=1}^k s_j(B)^{-2} \le C \frac{k}{i}.
  \]
  On the other hand, the left-hand side above is at least $s_{n-i}(B)^{-2}
  (k - n + i) \ge s_{n-i}(B)^{-2} i/2$. This gives that with probability at least $1 - 2n\exp(-c n^\gamma/K_n^2)$,
  \[
 n s_{n-i}^2 \ge s_{n-i}(B)^2 %
  \ge c \frac{ i^2 }{n - \lfloor n^{\gamma}/2\rfloor},
  \]
  which implies that $s_{n-i} \ge
  c\frac{i}{n}$. Taking another union bound
  over all $i \ge n^\gamma$ we obtain that
  $s_{n-i} \ge c\frac{i}{n}$ for all $n^\gamma \le i \le n-1$ with probability at least
  $1 - 2n^2\exp(-cn^\gamma/K_n^2)$. For some $\gamma \in (0,1)$, we have $\gamma - 2/(2+\delta) > 0$ and so
  by the assumption on $K_n$,
  \begin{displaymath}
    2n^2\exp(-c n^\gamma/K_n^2) \to 0,
  \end{displaymath}
  which ends the proof.
\end{proof}

\begin{proof}[Conclusion of the proof of Theorem \ref{th:circular_law}]
  Recall that we have to prove (ii). By Markov's inequality it suffices to
  show that for some $\alpha > 0$ and some constant $C_{z}$,
\begin{displaymath}
  \lim_{n\to\infty}
  \p\Big(\int_0^\infty (s^{\alpha} + s^{-\alpha})d\nu_{z,n}(s) > C_z\Big)=0.
\end{displaymath}
Note that, using the notation of Lemma \ref{le:intermediate_sv}, we have
\begin{displaymath}
  \int_0^\infty (s^{\alpha} + s^{-\alpha})d\nu_{z,n}(s) %
  = \frac{1}{n}\sum_{i=1}^n (s_i^\alpha + s_i^{-\alpha}).
\end{displaymath}
Note also that for all $\alpha \in (0,2]$ we have
\begin{align}\label{eq:positive_exponent}
  \Big(\frac{1}{n}\sum_{i=1}^n s_i^\alpha\Big)^{2/\alpha} %
  &\le \frac{1}{n}\sum_{i=1}^n s_i^2 \nonumber\\
  &= n^{-1}\|A^{(n)} - z\id_n\|_{\mathrm{HS}}^2\nonumber\\ %
  &\le 2|z|^2 + 2n^{-2}\sum_{i,j=1}^n |\bfx_{ij}^{(n)}|^2 \nonumber\\
  &= 2|z|^2+2.
\end{align}

As for the other sum, by the estimate of Theorem \ref{th:ssv_detailed} together with the assumption on $K_n$
we have with probability tending to one for some finite constant $\beta>0$,
$s_n \ge n^{-\beta}$.

Combining this with Lemma \ref{le:intermediate_sv} we get with probability
tending to one,
\begin{align*}
  \frac{1}{n} \sum_{i=1}^n s_{i}^{-\alpha} &= \frac{1}{n}\sum_{i = 0}^{\lfloor
    n^\gamma\rfloor } s_{n-i}^{-\alpha} %
  + \frac{1}{n} \sum_{i = \lfloor n^\gamma\rfloor +1}^{n-1} s_{n-i}^{-\alpha}\\
  &\le \frac{1}{n}n^{\beta\alpha} n^{\gamma} %
  + \frac{1}{n}C \sum_{i=\lfloor n^\gamma\rfloor +1}^{n-1} %
  \Big(\frac{n}{i}\Big)^\alpha \\
  &\le n^{\beta\alpha + \gamma - 1} + C_{\alpha} n^{\alpha - 1}
  n^{1-\alpha} \le C_{\alpha}
\end{align*}
for $\alpha$ small enough. Together with \eqref{eq:positive_exponent} this
gives (ii), and Theorem \ref{th:circular_law} is proved.
\end{proof}

\section{Proof of Theorem \ref{th:circular_law_2}\label{se:circular_law_2} (circular law for the second model)}

\begin{proof}[Proof of Theorem \ref{th:circular_law_2}]
  Let $d$ be any distance metrizing the weak convergence of probability
  measures on $\C$, such as the bounded-Lipschitz distance (also referred to
  as the Fortet-Mourier distance by some authors). Let $\mathcal{M}_n(\delta)$
  be the set of all $n\times n$ matrices $\bfx$, satisfying (A1-A2) with
  $K_n \le n^{1/(10+\delta/3)}$. Let us denote by
  $A^{(n)}(\bfx)=n^{-1/2}X^{(n)}(\bfx)$ the matrix constructed from $\bfx$ as
  in \eqref{eq:model}. Theorem \ref{th:circular_law} implies that for any
  $\delta > 0$ and $\varepsilon > 0$
  \begin{displaymath}
    \sup_{\bfx \in \mathcal{M}_n(\delta)} \p(d(\nu_{A^{(n)}(\bfx)},\nu^{\mathrm{circ}}) > \varepsilon) \to 0
  \end{displaymath}
  as $n\to \infty$. Indeed if for some $\varepsilon$ there exists a sequence
  $n_k \to \infty$ and $n_k\times n_k$ matrices $\bfx_{n_k}$ such that
  $\p(d(\nu_{A^{(n_k)}(\bfx_{n_k})},\nu^{\mathrm{circ}}) > \varepsilon) > \varepsilon$,
  then we can complete this sequence to a sequence $\bfx_n$, $n\ge 1$,
  violating Theorem \ref{th:circular_law}.

Consider now the matrices $B^{(n)}$ and for any $n$ define the event
\begin{displaymath}
\Delta_n = \{\max_{i,j\le n} \sigma_n^{-1}|Y_{ij}^{(n)} - \mu_n| \le n^{1/(10+\delta/3)}\}.
\end{displaymath}
Note that by the union bound, Markov's inequality and exchangeability
\begin{displaymath}
\p(\Delta_n') \le \frac{n^2 \sup_m \E\Big(\sigma_m^{-1}|Y^{(m)}_{11}-\mu_m|\Big)^{20+\delta}}{n^{\frac{20+\delta}{10+\delta/3}}} \to 0.
\end{displaymath}

Let $\pi_n$ be a random (uniform) permutation of $[n]\times[n]$, independent of $B^{(n)}$. Since
\begin{displaymath}
\bar{B}^{(n)} = \frac{1}{\sqrt{n}\sigma_n}(Y^{(n)}_{\pi_n(i,j)} - \mu_n)_{1\le i,j\le n}
\end{displaymath}
has the same distribution as $B^{(n)}$, it is enough to show that
$\nu_{\bar{B}^{(n)}}$ converges weakly in probability to the circular law. By the
Fubini theorem we have for any $\varepsilon > 0$,
\begin{align*}
\p(d(\nu_{\bar{B}^{(n)}},\nu^{\mathrm{circ}}) > \varepsilon) & \le \p(\Delta_n') + \E_B \Big(\p_{\pi_n}(d(\nu_{\bar{B}^{(n)}},\nu^{\mathrm{circ}}) > \varepsilon)\Ind{\Delta_n}\Big)\\
&\le \p(\Delta_n') + \sup_{\bfx \in \mathcal{M}_n(\delta)} \p(d(\nu_{A^{(n)}(\bfx)},\nu^{\mathrm{circ}}) > \varepsilon) \to 0.
\end{align*}
\end{proof}

%% Bibliography
\makeatletter
\def\@MRExtract#1 #2!{#1} %
\renewcommand{\MR}[1]{% we need to strip the "(...)"
  \xdef\@MRSTRIP{\@MRExtract#1 !}%
  \href{http://www.ams.org/mathscinet-getitem?mr=\@MRSTRIP}{MR-\@MRSTRIP}}
\makeatother
%\bibliographystyle{amsplain}
%\bibliography{cirex}

\providecommand{\bysame}{\leavevmode\hbox to3em{\hrulefill}\thinspace}

\end{document}